\title{\bf{On finite determinacy for matrices of power series}}
\author{
       \bf{ Gert-Martin Greuel and Thuy Huong Pham}\\
 }
\date{\today}

\documentclass[11pt]{article}
\usepackage{amsmath, amsthm, amssymb, mathrsfs, amscd, amsthm, amscd,amsfonts,graphicx}
\usepackage{indentfirst,url,xypic}
\usepackage{lipsum} 
\usepackage[all]{xy}
\usepackage{url}
\usepackage[colorlinks,plainpages]{hyperref}
\usepackage{verbatim}
\usepackage{fancyvrb}
\usepackage{listings}
\usepackage[colorlinks,plainpages]{hyperref}
\hypersetup{
	colorlinks=true,
	linkcolor=blue,
	filecolor=magenta,      
	urlcolor=cyan,
	citecolor=blue
}

\DeclareMathOperator{\rank}{rank}

\DeclareMathOperator{\variety}{V}

\DeclareMathOperator{\support}{Supp}

\DeclareMathOperator{\modular}{mod}
\DeclareMathOperator{\characteristic}{char}
\DeclareMathOperator{\order}{ord}
\DeclareMathOperator{\annihilator}{Ann}
\DeclareMathOperator{\spectrum}{Spec}

\newtheorem{Definition}{ Definition}[section]
\newtheorem{Theorem}[Definition]{Theorem}
\newtheorem{Remark}[Definition]{Remark}
\newtheorem{Proposition}[Definition]{Proposition }
\newtheorem{Corollary}[Definition]{Corollary}
\newtheorem{Example}[Definition]{Example}
\newtheorem{Lemma}[Definition]{Lemma}

\newcommand{\R}{\mathbb{R}}

\newcommand{\N}{\mathbb{N}}
\newcommand{\C}{\mathbb{C}}
\newcommand{\Z}{\mathbb{Z}}

\begin{document}
\maketitle

\begin{abstract}
Let $R=K[[x_1,...,x_s]]$ be the ring of formal power series with maximal ideal $\mathfrak{m}$ over a field $K$ of arbitrary characteristic. On the ring $M_{m,n}$ of $m\times n$ matrices $A$ with entries in $R$ we consider several equivalence relations given by the action on $M_{m,n}$ of a group $G$. $G$ can be the group of automorphisms of $R$, combined with the multiplication of invertible matrices from the left, from the right, or from both sides, respectively. We call $A$ finitely $G$-determined if $A$ is $G$-equivalent to any matrix $B$ with ${A-B} \in {\mathfrak{m}^k M_{m,n}}$ for some finite integer $k$, which implies in particular that $A$ is $G$--equivalent to a matrix with polynomial entries.

The classical criterion for analytic or differential map germs $f:(K^s,0) \to (K^m,0)$, $K = \R, \C$, says that $f \in M_{m,1}$ is finitely determined (with respect to various group actions) iff the tangent space to the orbit of $f$ has finite codimension in $M_{m,1}$. We extend this criterion to arbitrary matrices in $M_{m,n}$  if the characteristic of K is 0 or, more general, if the orbit map is separable. In positive characteristic however, the problem is more subtle since the orbit map is in general not separable, as we show by an example. This fact had been overlooked in previous papers. Our main result is a general sufficient criterion for finite $G$-determinacy in $M_{m,n}$ in arbitrary characteristic in terms of the tangent image of the orbit map, which we introduce in this paper. This criterion provides a computable bound for the $G$-determinacy of a matrix $A$ in $M_{m,n}$, which is new even in characteristic 0. \\

\end{abstract}

%%%%%%%%%%%%%%%%%%%%%%%%%%%%%%
%%%%%%%%%%%%%%%%%%%%%%%%%%%%%%
\section{Overview}
Throughout this paper let $K$ denote a field of arbitrary characteristic and 
$$R:=K[[{\bf{x}}]]=K[[x_1, x_2,..., x_s]]$$ 
the formal power series ring over $K$ with maximal ideal $\mathfrak{m}$. We denote by 
$$M_{m,n}:=Mat(m,n, R)$$
the ring of all $m\times n$ matrices of power series. We consider the  group of $K$-algebra automorphisms of $R$
$$\mathcal{R}:= Aut(R)$$
and the semi-direct products
\begin{align*}
&{\mathcal{G}}_l:=GL(m,R)\rtimes{\mathcal{R}},\\
&{\mathcal{G}}_r:=GL(n,R)\rtimes{\mathcal{R}}, {\text{and}}\\
&\mathcal{G}_{lr}:=\left(GL(m,R)\times GL(n,R)\right)\rtimes \mathcal{R}.
\end{align*}
These groups act on the space $M_{m,n}$ as follows
\begin{align*} 
&(\phi, A)\mapsto \phi (A):=[\phi(a_{ij}({\bf{x}}))]=[a_{ij}(\phi ({\bf{x}}))],\\
%%%%%%%%%%%%%%%%
&(U,\phi, A)\mapsto U\cdot\phi (A)=U\cdot[\phi(a_{ij}({\bf{x}}))]=U\cdot[a_{ij}(\phi ({\bf{x}}))],\\
%%%%%%%%%%%%%%%%%%%% 
&(V,\phi, A)\mapsto \phi (A)\cdot V=[\phi(a_{ij}({\bf{x}}))]\cdot V=[a_{ij}(\phi ({\bf{x}}))]\cdot V, {\text{and}}\\
%%%%%%%%%%%%%%%%%%
&(U,V,\phi, A)\mapsto U\cdot\phi (A)\cdot V=U\cdot[\phi(a_{ij}({\bf{x}}))]\cdot V=U\cdot[a_{ij}(\phi ({\bf{x}}))]\cdot V,
\end{align*}
%%%%%%%%%%%%%%%%%%
where ${\bf{x}}=(x_1, x_2,\ldots, x_s)$, $A=[a_{ij}({\bf{x}})]\in M_{m,n}$, $U\in GL(m,R)$, $V\in GL(n,R)$, and $\phi ({\bf{x}}):=(\phi_1,\ldots,\phi_s)$ with $\phi_i:=\phi(x_i)\in\mathfrak{m}$ for all $i=1,\ldots,s$.

\vskip 7pt Throughout this paper let $G$ denote one of the groups $\mathcal{R}$, ${\mathcal{G}}_l$, ${\mathcal{G}}_r$, and ${\mathcal{G}}_{lr}$. \\

For $A\in M_{m,n}$, we denote by $GA$ the orbit of $A$ under the action of $G$ on $M_{m,n}$. Two matrices $A, B \in M_{m,n}$ are called  {\it $G$-equivalent}, denoted  $A\mathop\sim\limits^{G} B$, if $B\in GA$.  A matrix $A\in  M_{m,n}$ is said to be {\it $G$ $k$-determined} if for each matrix $B\in  M_{m,n}$ with $B-A\in \mathfrak{m}^{k+1}\cdot M_{m,n}$, we have $B\mathop\sim\limits^{G} A$, i.e. if $A$ is $G$-equivalent to every matrix which coincides with $A$ up to and including terms of degree $k$. $A$ is called {\it finitely $G$-determined} if there exists a positive integer $k$ such that it is $G$ $k$-determined. 

\vskip 7pt Note that the case $n=1$, i.e. $M_{m,1}$, covers the case of map-germs $(f_1,...,f_m)$, ${K[[y_1,...,y_m]] \to R}$, $y_i \mapsto f_i$, where $G$-equivalence is called {\it right-equivalence} for $G = \mathcal{R}$ and {\it contact-equivalence} for $G = {\mathcal{G}}_l$; the case $m=n=1$ is the classical case of one power series. In \cite{GP16} we give necessary and sufficient conditions for finite determinacy of map germs in arbitrary characteristic, in particular for complete intersections, also for non-separable orbit maps. 

\vskip 7pt Over the real and complex numbers $K$, finite determinacy was studied for $M_{m,1}$, i.e. for differentiable and analytic map-germs $(f_1,...,f_m):(K^m,0) \to (K^s,0)$, by \cite{Tou68}, \cite{Mat68}, \cite{Wal81}, \cite{Gaf79}, \cite{Ple80}, \cite{Dam81}, \cite{BdPW87}, \ldots.  In \cite{BK16}, the authors study finite determinacy for matrices of power series in $M_{m,n}$ over fields of characteristic $0$ with respect to various equivalence relations. As the methods of proof usually involve integration of vector fields, they can not be transferred to positive characteristic. Moreover, the starting point of all previous investigations was the fact that the tangent space at a point of the orbit is equal to the image of the tangent space of the group at the identity under the orbit map. This is in general no more true in positive characteristic as we show in this paper.  The case of one power series, i.e. $M_{1,1}$, over a field of arbitrary characteristic was treated in \cite{GrK90} for contact equivalence and in \cite{BGM12} for right and contact equivalence. The present paper is an extension of some results of the PhD thesis of the second author,  see \cite{P16}.

%\vskip 7pt 
\indent We give a short overview of the results of this paper: In section \ref{tangent image} we introduce the {\it tangent image} $\widetilde T_A(GA)$ at a matrix $A$ to the orbit $GA$ to be the inverse limit of the images of the tangent maps to 
	\begin{align*} 
	o^{(k)}: G^{(k)}&\to G^{(k)}jet_k(A),
	\end{align*} 
where $o^{(k)}$ is the induced map by restricting the orbit map $G \to GA$ to the jet space of power series up to order $k$. In Proposition {\ref {tangent space}} and Definition {\ref {tangent space for groups}} we give an explicit and computable description of $\widetilde T_A(GA)$ in terms of the entries and the partials of the entries of $A$.\

In section \ref{sufficient condition 1}, Theorem {\ref{finite determinacy}}, we prove our main result:\

\vskip 7pt
{\bf Theorem:} 
	{\it
Let $A \in \mathfrak{m}\cdot M_{m,n}$. If 
         $$\dim_{K}\left( M_{m,n}/ \widetilde T_A(GA)\right)<\infty,$$
then $A$ is finitely $G$-determined. More precisely, if there is some integer $k\ge 0$ such that

        $$ \mathfrak{m}^{k+2}\cdot M_{m,n}\subset \mathfrak{m}\cdot \widetilde T_A(GA),$$
	then A is G $(2k-\order(A)+2)$-determined.
}
\vskip 7pt

Here, for $f\in R$, $f\ne 0$, we denote by $\order(f)$ the order of $f$, i.e. the maximal positive integer $l$ such that $f\in \mathfrak{m}^l$, and for $f=0$, we set $\order(f)=\infty$. For $A=[a_{ij}]\in M_{m,n}$, we set $\order(A):=\min\{\order(a_{ij})\}.$

\vskip 7pt
In section \ref{sufficient condition 2}, we formulate several other equivalent sufficient conditions for finite $G$-determinacy and prove alternative determinacy bounds. If the orbit map $o^{(k)}$ is separable for sufficiently big $k$, these conditions are even equivalent to finite determinacy as we prove in Theorem \ref{separable criterion}:  

\vskip 7pt
{\bf Theorem:} 
{\it
Assume there is some $k\in \N$ such that the orbit map  $G^{(l)}\to G^{(l)}jet_l(A)$ is separable for all $l\ge k$ (e.g. if $\characteristic(K)=0$). Then $A$ is finitely $G$-determined if and only if $\widetilde T_A(GA)$ has finite codimension in $M_{m,n}$.
}

\vskip 7pt
This follows since for a separable orbit map the tangent map is surjective and hence the tangent image to the orbit coincides with the tangent space to the orbit. We show by an example (cf. Example \ref{surjective tangent map}) that even in the simplest case of one function $f$ the orbit map need not be separable if $K$ has positive characteristic, also for $f$ being finitely determined. This fact had been overlooked in previous papers and came as a surprise to us, see Remark \ref{gap}. Finally we apply the above result to classify finitely $\mathcal{R}$-determined matrices $A \in M_{m,n}$ (under the separability condition) and show that for $m>1$ finite $\mathcal{R}$-determinacy holds only in the non--singular case.

% % % % % % % % % % % % % % % % % % % % % % % % % % % % % % % % % % % %
\section{The tangent image of the orbit map}\label{tangent image}

 In this section, we identify the images of the tangent maps induced by the orbit maps. Since the power series ring $R$ and the group $G \in \left\{\mathcal{R}, \mathcal{G}{_l}, \mathcal{G}{_r}, \mathcal{G}_{lr}\right\}$ are infinite dimensional over $K$ we pass, as usual, to the space of k-jets.

For $A\in M_{m,n}$ and $k\in \mathbb{N}$, we denote by 
$jet_k(A)$ the image of $A$ in $M_{m,n}/\mathfrak{m}^{k+1}\cdot M_{m,n}$, the {\textit{ k-jet}} of $A$, and by
	$$M_{m,n}^{(k)}:=M_{m,n}/\mathfrak{m}^{k+1}\cdot M_{m,n},$$ 
the space of all $k$-jets. The k-jet of $G$ is
	$$G^{(k)}:=\{jet_k(g)\mid g\in G\},$$
where $jet_k(g)=\left(jet_k(U), jet_k(V), jet_k(\phi)\right)$, 
$jet_k(\phi)(x_i)=jet_k(\phi(x_i))$, for $g=(U,V,\phi)\in \mathcal{G}_{lr}$ and similar for $G=\mathcal{R}$, $\mathcal{G}_l$, $\mathcal{G}_r$. Then $G^{(k)}\in\left\{\mathcal{R}^{(k)}, \mathcal{G}_l^{(k)}, \mathcal{G}_r^{(k)}, \mathcal{G}_{lr}^{(k)}\right\}$ is an affine algebraic group with group structure given by $jet_k(g)jet_k(h) = jet_k(gh)$, acting algebraically on the affine space $M_{m,n}^{(k)}$ via
\begin{align*}
	&G^{(k)}\times M_{m,n}^{(k)}\to M_{m,n}^{(k)}, \hskip4pt
	\left(jet_k(g), jet_k(A)\right)\mapsto jet_k(gA),
\end{align*}
i.e. we let representatives act and then take the $k$-jets. Everything is defined over $K$. 

\begin{Remark} \label{field extension} \rm
For a geometric interpretation of the orbit map and since we are going to apply results about algebraic group actions wich are formulated for algebraically closed fields, we fix an algebraically closed extension field $K'$ of $K$, e.g. an algebraic closure of $K$.

(1) Let $H$ be an algebraic group defined over $K$ acting $K$--algebraically on the algebraic $K$--variety $X$. Then $X$ resp. $H$ defines an algebraic variety $X'$ resp. an algebraic group $H'$ over $K'$ and the action of $H$ on $X$ extends naturally to an action of $H'$ on $X'$. $X'$ and $H'$ are $K'$--varieties, i.e. schemes of finite type over $K'$, which are defined over $K$, as well as the action $H'\times X' \to X'$ (in the sense of Borel, cf. [Bor 91], i.e. given by polynomial data, where the polynomials have coefficients in $K$, with points being closed points). For $x\in X$ (i.e. a $K$--rational point of $X'$) the orbit $H'x$ is a subvariety of $X'$ defined over $K$ and the orbit map $o':H'\to X',  h \to hx,$ is also defined over $K$.

(2) Recall that for $X$ an algebraic $K$--variety with structure sheaf $\mathcal{O}_X$, the Zariski-tangent space $T_xX$ of $X$ at the point $x\in X$ can be described as
$$T_xX:=\left\{K{\text{-algebra homomorphisms}}\hskip 4pt  f: \mathcal{O}_{X,x}\to K[\epsilon] \hskip 4pt\big| \hskip 4pt p\circ f=\chi_x \right\},$$
where $\mathcal{O}_{X,x}$ is the local ring of $X$ at  $x$,   $K[\epsilon]$ the ring of dual numbers $K[\epsilon]=K[t]/\langle t^2\rangle=\{a+b\epsilon\hskip 4pt| \hskip 4pt a, b\in K, \epsilon^2=0\},$
and $p$ and $\chi_x$ are the canonical residue maps $K[\epsilon]\to K[\epsilon]/\langle \epsilon \rangle= K$ and $\mathcal{O}_{X,x}\to \mathcal{O}_{X,x}/\mathfrak{m}_x=K$, where $\mathfrak{m}_x$ is the maximal ideal of $\mathcal{O}_{X,x}$.

In the same way we have the tangent space $T_xX'$ defined w.r.t. $K'$, satisfying 
\[
T_xX'=T_xX\otimes_K K'.
\]

(3) The tangent map $To': T_eH' \to T_xX'$, $e\in H$ the identity element, of the orbit map $o'$ is defined over $K$ (induced by $\mathcal{O}_{X,x} \to \mathcal{O}_{H,e}$) and hence induces a map $T_eH\to T_xX$ which we denote by $To$, satisfying $To'=To\otimes_K K'$. We define
\[
\begin{array}{lcl}
\widetilde{T}_x(Hx): & = & \text {im} (To: T_eH\to T_xX),\\
\widetilde{T}_x(H'x): & = & \text {im} (To': T_e H'\to T_xX')
\end{array}
\]
and call it the \textit{tangent image} at $x$ of the orbit map. Obviously we have
\[
\widetilde{T}_x(H'x)=\widetilde{T}_x(Hx)\otimes_K K'.
\]
\end{Remark}

The following proposition is well--known (cf. \cite{WA05} Theorem 3.1).

\begin{Proposition}\label{Proposition2.3}
Let $K'$ be an algebraically closed field, $X'$ an algebraic $K'$--variety and  $H'$ an algebraic group over $K'$ acting algebraically on $X'$. For $x \in X'$ the orbit map $o':H'\to H'x$ is separable iff the tangent map $To': T_eH'\to T_x(H'x)$ is surjective. 
\end{Proposition}

Recall that the orbit map is called separable if the extension of fields of rational functions $K(H'x)\subset K(H')$ is separably generated, i.e. there is a transcendence base $\{x_i\}$ of $K(H'x)\subset K(H')$ such that $K(H')$ is a separable algebraic extension of $K(H'x)(\{x_i\})$. Separability holds e.g. if $\characteristic(K')=0$.

\begin{Corollary} \label{Corollary2.3}
Let $K$ be a field, $X$ an algebraic $K$--variety and  $H$ an algebraic group over $K$ acting algebraically on $X$. Let $K'$ be an algebraically closed extension field of $K$ and let $X'$ and $H'$ be as in Remark \ref {field extension}. For $x \in X$  the following are equivalent:
\begin{enumerate}
\item [(i)] The orbit map $o': H'\to H'x$ is separable.
\item [(ii)] $\dim_K \widetilde{T}_x(Hx)=\dim_x H'x$
\item [(iii)] $\widetilde{T}_x (Hx)=T_x(Hx)$
\end{enumerate}
These conditions hold in particular if $\characteristic(K)=0$.
\end{Corollary}

\begin{proof}
Since the orbit $H'x$ is smooth we have $\dim_x H'x=\dim_{K'} T_x(H'x)$ and the latter equals $\dim_K T_x(Hx)$ by Remark \ref {field extension}. This shows the equivalence of (ii) and (iii). The equality in (iii) is equivalent to $\widetilde{T}_x(H'x)=T_x(H'x)$ and hence (iii) is equivalent to (i) by Proposition \ref{Proposition2.3}.
\end{proof}

\vskip 7pt

The following easy lemma is used below; it replaces the Taylor series in positive characteristic.
\begin{Lemma}\label{expansion}
	Let $f({\bf{x}})=
	\sum\limits_{\left| {\bf{\alpha}}  \right| \ge ord(f)} {c_\alpha  \bf x^{\bf{\alpha}}}\in K[[{\bf x}]]$ and ${\bf z}=(z_1,\ldots,z_s)$ new variables. Then
	\begin{align*}
	f({\bf x+z})=f({\bf x})+
	\sum\limits_{\nu = 1}^s {\frac{{\partial f({\bf x})}}{{\partial x_\nu }}\cdot z_\nu } +
	\sum\limits_{\left| \alpha  \right| \ge \order(f)} {c_\alpha\cdot  \left( {\sum\limits_{\left| \gamma  \right| \ge 2\atop 
				\gamma  \le \alpha } {\binom
				{\alpha _1 } { \gamma _1} \cdot\ldots\cdot\binom
				{\alpha _s } { \gamma _s}{\bf x}^{\alpha  - \gamma }{\bf z}^\gamma} }\right)},
	\end{align*}
	where $\gamma\le \alpha$ means that $\gamma_\nu\le\alpha_\nu$ for all $\nu=1,\ldots,s$, $\binom {\alpha _\nu } { \gamma _\nu}\in \Z$ for all $\nu=1,\ldots,s$ and if $\characteristic(K)=p>0$  we denote for $k\in \Z$
	$$k{\bf x}^{\alpha  - \gamma }{\bf z}^\gamma= 
	\begin{cases} 0 & \mbox{if }  p\mid k \\ 
	k {(\modular\hskip 3pt p)}{\bf x}^{\alpha  - \gamma }{\bf z}^\gamma & \mbox{if }p\nmid k 
	\end{cases} $$
\end{Lemma}

\vskip 7pt

Denote by 
$$E_{m,pq}\hskip 7pt (\text{resp.} \hskip 4pt  E_{n,hl}) $$
the $(p, q)$-th (resp. $(h, l)$-th) {{canonical matrix}} of the ring of square matrices $M_m:=Mat(m,m,R)$  (resp. $M_n:=Mat(n,n,R)$) with 1 at the place $(p,q)$ (resp. $(h,l)$) and 0 else. For $A=[a_{ij}({\bf x})]\in M_{m,n}$ we set $\frac{\partial A}{\partial x_\nu}=\left[\frac{\partial a_{ij}({\bf x})}{\partial x_\nu}\right]\in M_{m,n}$.

\begin{Proposition}{\label {tangent space}}
	Let $A\in M_{m,n}$ and $k\ge 1$. Let $G=\mathcal{G}_{lr}$ and let $e=(I_m, I_n, id_R)\in G^{(k)}$ be the identity of the group $G^{(k)}$. Then the tangent image, i.e. the image of the tangent map
	$$T_eG^{(k)}  \to T_{jet_k(A)}\left( G^{(k)}jet_k(A)\right),$$
	considered as a subspace of $ M_{m,n}^{(k)}$, is the submodule
	\begin{align*}
	&\hskip 20pt\widetilde T_{jet_k(A)}\left(G^{(k)}jet_k(A)\right):=\\
	&\hskip 10pt \left(\langle E_{m, pq}\cdot A\rangle +\langle A \cdot E_{n, hl}\rangle+ \mathfrak{m}\cdot\left\langle\frac{\partial A}{\partial x_\nu}\right\rangle+\mathfrak{m}^{k+1}\cdot M_{m,n}\right){\Big/}\mathfrak{m}^{k+1}\cdot M_{m,n},
	\end{align*}
	where $\langle E_{m, pq}\cdot A\rangle$, $\langle A\cdot E_{n, hl}\rangle$, and  $\left\langle\frac{\partial A}{\partial x_\nu}\right\rangle$ are the $R$-submodules of $M_{m,n}$ generated by $E_{m, pq}\cdot A$, $p, q=1,\ldots,m$, $A\cdot E_{n, hl}$, $h, l=1,\ldots,n$, and $\frac{\partial A}{\partial x_\nu}$, $\nu=1,\ldots,s$, respectively. 
	
Moreover, the tangent images of $\mathcal{R}$, $\mathcal{G}_l$, and $\mathcal{G}_r$ are respectively
		\begin{align*}
		&\widetilde T_{jet_k(A)}\left(\mathcal{R}^{(k)}jet_k(A)\right):= \left(\mathfrak{m}\cdot\left\langle\frac{\partial A}{\partial x_\nu}\right\rangle+\mathfrak{m}^{k+1}\cdot M_{m,n}\right){\Big/}\mathfrak{m}^{k+1}\cdot M_{m,n},\\
		&\widetilde T_{jet_k(A)}\left(\mathcal{G}_l^{(k)}jet_k(A)\right):=
		\left(\langle E_{m, pq}\cdot A\rangle + \mathfrak{m}\cdot\left\langle\frac{\partial A}{\partial x_\nu}\right\rangle+\mathfrak{m}^{k+1}\cdot M_{m,n}\right){\Big/}\mathfrak{m}^{k+1}\cdot M_{m,n},\\
		&\widetilde T_{jet_k(A)}\left(\mathcal{G}_r^{(k)}jet_k(A)\right):=
		\left(\langle A \cdot E_{n, hl}\rangle + \mathfrak{m}\cdot\left\langle\frac{\partial A}{\partial x_\nu}\right\rangle+\mathfrak{m}^{k+1}\cdot M_{m,n}\right){\Big/}\mathfrak{m}^{k+1}\cdot M_{m,n}.
		\end{align*}
\end{Proposition}

\begin{proof}
	The orbit map 
	\begin{align*} 
	o^{(k)}: G^{(k)}&\to G^{(k)}jet_k(A),\\
	\left(jet_k(U), jet_k(V),jet_k(\phi)\right)&\mapsto jet_k\left(U\cdot\phi(A)\cdot V\right)=jet_k\left(U\cdot[a_{ij}(\phi({\bf x}))]\cdot V\right),
	\end{align*} 
	where $A=[a_{ij}({\bf x})]$, induces the tangent map 
	$$T_eG^{(k)}  \to T_{jet_k(A)}\left(G^{(k)}jet_k(A)\right).$$ 
	Each element of $T_eG^{(k)}$ can be represented by a triple
	$$\left(jet_k(I_n+\epsilon\cdot U), jet_k(I_m+\epsilon\cdot V), jet_k (id_R+\epsilon\cdot \phi)\right),$$
	where $U\in M_m$,  $V\in M_n$, and $\phi(x_\nu)=:\phi_\nu\in \mathfrak{m}$ for all $\nu=1,\ldots, s$.
	Letting this triple act on $jet_k(A)$  we get 
	$$jet_k\left((I_m+\epsilon\cdot U)\cdot[a_{ij}({\bf x}+\epsilon\cdot\phi({\bf x}))]\cdot(I_n+\epsilon\cdot V)\right),$$
	where $\phi({\bf x}):=(\phi(x_1),\ldots,\phi(x_s))=(\phi_1,\ldots, \phi_s)$. 
	
	Now for all $i=1,\ldots, m$ and $j=1,\ldots, n$, let $o_{ij}:=\order(a_{ij})$ and apply Lemma \ref{expansion} to 
	$$f=a_{ij}
	= \sum\limits_{\left| \alpha  \right| \ge o_{ij}} {c_\alpha ^{(ij)} {\bf x}^\alpha  }, \hskip 7pt  {\bf z}=\epsilon\cdot \phi({\bf x}).$$
	Since $\epsilon ^2=0$ we have
	\begin{align*}
	a_{ij}({\bf x}+{\bf z})=a_{ij}({\bf x})+
	\epsilon\cdot\sum\limits_{\nu = 1}^s {\frac{{\partial a_{ij}({\bf x})}}{{\partial x_\nu}}\cdot\phi_\nu}.
	\end{align*}
	This implies
	$$(I_m+\epsilon\cdot U)\cdot[a_{ij}({\bf x}+\epsilon\cdot\phi({\bf x}))]\cdot(I_n+\epsilon\cdot V)=A+\epsilon\cdot\left(A\cdot V+U\cdot A+\sum\limits_{\nu = 1}^s {\phi _\nu\cdot \frac{{\partial A}}{{\partial x_\nu }}}\right)$$
	so that the image of the triple under the tangent map is the $k$-jet of the matrix 
	$$A+\epsilon\cdot\left(A\cdot V+U\cdot A+\sum\limits_{\nu = 1}^s {\phi _\nu\cdot \frac{{\partial A}}{{\partial x_\nu }}}\right).$$
	Hence, the claim follows for $\mathcal{G}_{lr}$. The proofs for the other groups are similar.
\end{proof}

From now on, for $A\in M_{m,n}$, we use the notations $\langle E_{m, pq}\cdot A\rangle$, $\langle A\cdot E_{n, hl}\rangle$, and  $\left\langle\frac{\partial A}{\partial x_\nu}\right\rangle$ as in Proposition {\ref {tangent space}}.

\begin{Definition}{\label{tangent space for groups}}
	For $A\in M_{m,n}$, we call the $R$-submodules of $M_{m,n}$
    \begin{align*}
		&\widetilde {T}_A(\mathcal{R}A):= \mathfrak{m}\cdot\left\langle\frac{\partial A}{\partial x_\nu}\right\rangle,\\
		&\widetilde T_A(\mathcal{G}_lA):= \langle E_{m, pq}\cdot A\rangle +\mathfrak{m}\cdot\left\langle\frac{\partial A}{\partial x_\nu}\right\rangle,\\
		&\widetilde T_A(\mathcal{G}_rA):=\left \langle A\cdot E_{n, hl}\right\rangle + \mathfrak{m}\cdot\left\langle\frac{\partial A}{\partial x_\nu}\right\rangle,{\text and} \\
		&\widetilde T_A(\mathcal{G}_{lr}A):=\langle E_{m, pq}\cdot A\rangle +\langle A\cdot E_{n, hl}\rangle+ \mathfrak{m}\cdot\left\langle\frac{\partial A}{\partial x_\nu}\right\rangle
	\end{align*}
the {\bf\textit{tangent images}} at $A$ to the orbit of $A$ under the actions of $\mathcal{R}$, $\mathcal{G}_{l}$, $\mathcal{G}_{r}$, and $\mathcal{G}_{lr}$ on $M_{m,n}$, respectively. \\
Replacing $\mathfrak{m}\cdot\left\langle\frac{\partial A}{\partial x_\nu}\right\rangle$ by $R\cdot\left\langle\frac{\partial A}{\partial x_\nu}\right\rangle$ in the above definition we get the {\bf\textit {extended tangent images}} $\widetilde {T}^e_A(\mathcal{R}A)$, $\widetilde T^e_A(\mathcal{G}_lA)$, $\widetilde T^e_A(\mathcal{G}_rA)$, and $\widetilde T^e_A(\mathcal{G}_{lr}A)$ .
\end{Definition}

\vskip 7pt

Note that $\{\widetilde T_{jet_k(A)}\left(G^{(k)}jet_k(A)\right), \pi_k\}_k\subset M_{m,n}^{(k)}$ is an inverse system of $R$-modules, where $\pi_k$ is induced by the canonical projection $M_{m,n}^{(k)} \to M_{m,n}^{(k-1)}$, and we have
		$$\widetilde T_A(GA)=	
		\mathop {\lim }\limits_{\mathop {\longleftarrow} \limits_{k\ge 0} } \widetilde T_{jet_k(A)}\left(G^{(k)}jet_k(A)\right)\subset M_{m,n}.
		$$
		
Likewise, $\{T_{jet_k(A)}\left(G^{(k)}jet_k(A)\right), \pi_k\}_k\subset M_{m,n}^{(k)}$ is an inverse system of $K$-vector spaces and we call the $K$-vector space
		$$T_A(GA):=	
		\mathop {\lim }\limits_{\mathop {\longleftarrow} \limits_{k\ge 0} } T_{jet_k(A)}\left(G^{(k)}jet_k(A)\right)\subset M_{m,n}.
		$$
		the	{\bf tangent space} at $A$ to the orbit $GA$.

\begin{Remark}\rm	
We have inclusions $\widetilde T_{jet_k(A)}\left(G^{(k)}jet_k(A)\right) \subset T_{jet_k(A)}\left(G^{(k)}jet_k(A)\right)$ for the $k$-jets  and hence $\widetilde T_A(GA)\subset T_A(GA)$ with equality if the $k$-jets coincide for sufficiently big $k$. Below we give an example with $\widetilde T_A(GA) \ne T_A(GA)$.
\end{Remark}

Using Corollary \ref {Corollary2.3} we get

\begin{Lemma} \label{tangent space char 0} 
$\widetilde T_A(GA) = T_A(GA)$ if there is an integer $k$ such that the orbit map $o^{(l)}:G^{(l)}\to G^{(l)}jet_l(A)$ (over an algebraic closure of K) is separable for all $l\ge k$ (e.g. if $\characteristic(K)=0$). Conversely, $\widetilde T_A(GA) = T_A(GA)$ implies that $o^{(l)}$ is separable for all $l$.
 \end{Lemma}

In the following, when we say that the orbit map is separable, we mean that this holds over an algebraic closure of $K$. Also the dimension of a $K$--variety is its dimension over an algebraic closure. 

\begin{Example}\label{surjective tangent map}\rm
We give examples where the tangent image is strictly contained in the tangent space, i.e. the tangent map 
$$T_eG^{(k)}  \to T_{jet_k(A)}\left(G^{(k)}jet_k(A)\right)\label{tangent map}$$
is not surjective and hence the orbit map $G^{(k)}\to G^{(k)}jet_k(f)$ is not separable. 

	\begin{enumerate}

		\item Let $\characteristic(K)=p>0$ and $k\ge p$. Let the right group $G=\mathcal{R}$ act on $K[[x_1,\ldots, x_s]]$ and let $f=x_1^p+\ldots+x_s^p$. Then ${\rm{j}}(f)=0$ so that the tangent image  $\mathfrak{m} \cdot {\rm{j}}(f)$ is zero. On the other hand, for $\phi\in \mathcal{R}$, $\phi=(ax_1, x_2,\ldots, x_s)$ where $a\ne 0$ and $a^p-1\ne 0$, we have
	$$jet_k(f\circ\phi)=a^px_1^p+x_2^p+\cdots+x_s^p\ne x_1^p+x_2^p+\cdots+x_s^p=jet_k(f),$$
showing that the orbit of $jet_k(f)$ and hence its tangent space has dimension $\ge 1$. Note that $f$ is not finitely $\mathcal{R}$-determined.

		\item Let $\characteristic(K)=2$, $f = x^2 + y^3$ and let the contact group $G=\mathcal{G}_l$ act on $R=K[[x,y]]$. We compute:

\begin{itemize}
	\item $f$ is $\mathcal{G}_l$ 4-determined

	\item the tangent image $\widetilde T_{f}(\mathcal{G}_l f) = \langle f \rangle + \mathfrak{m} \cdot {\rm{j}}(f)$ is $\langle x^2,xy^2,y^3 \rangle$

	\item its 4-jet $\widetilde T_{jet_4(f)}\left(\mathcal{G}_l^{(4)}  jet_4(f)\right)$ has dimension 10 in $R^{(4)} = R/ \mathfrak{m}^5$

	\item the group $\mathcal{G}_l^{(4)}$ has dimension 43 and the stabilizer of $jet_4(f)$ has dimension 32 in  $R^{(4)}$.

\end{itemize}

It follows that the orbit $\mathcal{G}_l^{(4)} jet_4(f)$ has dimension 11 in  $R^{(4)}$. Since the tangent image has dimension 10, $\widetilde T_{f}(\mathcal{G}_l  f) \ne T_{f}(\mathcal{G}_l  f)$ by Corollary \ref {Corollary2.3}. 

\item  Similar examples exist in other positive characteristics. E.g. we have in characteristic 3 for $f = x^3 + y^4$ that the dimension in $R^{(5)}$ of the $\mathcal{G}_l$--tangent image is 11 while that of the tangent space is 12. 

On the other hand we can show that for $f = x^2 + y^3$ and $G=\mathcal{G}_l$ the tangent image coincides with the tangent space in characteristic 3 and for $G=\mathcal{R}$ also in characteristic 2, implying that in these cases the orbit map is separable. 
	
	\end{enumerate}
	
\end{Example}

The computations were done by using SINGULAR (cf. \cite {DGPS15}). The determinacy (using Theorem {\ref{finite determinacy}}) and the tangent image of the orbit map are easily computable by using standard bases in the local ring R (see \cite {GP07}). To compute the tangent space to the orbit we need to compute equations of the orbit, which is possible but much harder, as in general a large number of group variables has to be eliminated. The dimension of the orbit is usually easier obtained by computing the stabilizer. Note that the dimension of a variety computed by standard resp. Gr\"obner bases refers to the dimension over the algebraic closure. Details of the computational part will appear elsewhere.

\begin{Remark} \label{gap} \rm The examples above show that the statement of Proposition 1 in \cite {BGM12} is wrong in positive characteristic for the right group as well as for the contact group, for the latter even if $f$ is finitely determined. Hence the proof of the ``if" part of Theorem 5 in \cite {BGM12} contains a gap. However, this gap can be closed as we show in our paper \cite {GP16} by proving a more general theorem for complete intersections.

\end{Remark}

We finish this section with an obvious necessary condition for finite $G$-determinacy:

\begin{Lemma} If $A$ is finitely $G$-determined then
		$$\dim_K M_{m,n}/T_A(GA)<\infty.$$
 \end{Lemma}
		Indeed, let $A$ be $k$-determined, $l>k$, and $B\in \mathfrak{m}^{l}\cdot M_{m,n}$. Then $A+tB\in GA$ for all $t\in K$. This yields $jet_l(B)\in T_{jet_l(A)}\left(G^{(l)}jet_l(A)\right)$ so that
		$B\in T_A(GA).$
		This means that $\mathfrak{m}^l\cdot M_{m,n}\subset T_A(GA)$, and thus the claim follows.

% % % % % % % % % % % % % % % % % % % % % % % % % % % % % % % % % % % % 
\section{A sufficient condition in terms of the tangent image}\label{sufficient condition 1}

In this section we establish a sufficient condition for finite G-determinacy of matrices in terms of the tangent image defined in section \ref{tangent image}.

\begin{Lemma}{\label{lemma}}
	Let $A, B\in M_{m,n}$.
	\begin{enumerate}
		\item\label{lemma1} If $A\mathop  \sim \limits^{\mathcal{R}} B$, i.e. 
		there is an automorphism $\phi\in Aut(R)$ such that $B=\phi(A)$, then, for the submodules of $M_{m,n}$ generated by the partials, we have 
		$$\widetilde T^e_B(\mathcal{R}B)=\phi \left(\widetilde T^e_A(\mathcal{R}A)\right).$$
		\item If $A\mathop  \sim \limits^{\mathcal{G}_{lr}}B$, i.e. there are  invertible matrices $U\in GL(m,R)$,  $V\in GL(n,R)$, and  an automorphism $\phi\in Aut(R)$ such that $B=U\cdot\phi(A)\cdot V$, then 
		\begin{align*} 
		&\widetilde T^e_B(\mathcal{G}_{lr}B)=U\cdot\phi \left(\widetilde T^e_A(\mathcal{G}_{lr}A)\right)\cdot V.
		\end{align*} 
		The same holds for $\widetilde T_B(\mathcal{G}_{lr}B)$ and $\widetilde T_A(\mathcal{G}_{lr}A)$ instead of $\widetilde T^e_B(\mathcal{G}_{lr}B)$ and  $\widetilde T^e_A(\mathcal{G}_{lr}A)$. A similar result holds for $\mathcal{G}_{l}$ and $\mathcal{G}_{r}$.
	\end{enumerate}
\end{Lemma}

\begin{proof}
The proof of 1. is an application of the chain rule to the entries of $B$ and for 2. use in addition the product rule.
\end{proof}

Our main result is the following theorem.

\begin{Theorem}{\label{finite determinacy}}
	Let $A=[a_{ij}]\in \mathfrak{m}\cdot M_{m,n}$ and let  $o:=\order(A)$. If there is some integer $k\ge 0$ such that
	\begin{align}
	\mathfrak{m}^{k+2}\cdot M_{m,n}\subset \mathfrak{m}\cdot \widetilde T_A(GA)\label{main},
	\end{align}
	then A is G $(2k-o+2)$-determined.
\end{Theorem}

\begin{proof}
	If $A$ is the zero matrix, it is not finitely $G$-determined and the statement is true. Hence, we may assume that $A\ne 0$. Assume that
	$$\mathfrak{m}^{k+2}\cdot M_{m,n}\subset \mathfrak{m}\cdot\langle E_{m, pq}\cdot A\rangle+\mathfrak{m}\cdot\langle A\cdot E_{n, hl}\rangle+\mathfrak{m}^2\cdot\left\langle \frac{\partial A}{\partial x_\nu}\right\rangle.$$

	For all $i=1,\ldots,m$ and $j=1,\ldots,n$, set $o_{ij}:=\order(a_{ij})$. Choose  $h\in\{1,\ldots,m\}$ and  $l\in\{1,\ldots,n\}$ such that $o=\order(a_{hl})$. Looking at $E_{hl}$ in $M_{m,n}$, the hypothesis implies that

	$$\mathfrak{m}^{k+2}\subset\mathfrak{m}\cdot\left\langle a_{1l}, a_{2l},\ldots, a_{ml}, a_{h1}, a_{h2},\ldots, a_{hn}\right\rangle+\mathfrak{m}^2\cdot {\rm j}(a_{hl})\subset \mathfrak{m}^{o+1}.$$

	This yields $k\ge o-1$. Set $N:=2k-o+2\ge k+1$. Let $B\in M_{m,n}$ be such that $B-A\in \mathfrak{m}^{N+1}\cdot M_{m,n}$. We show that $B\mathop  \sim \limits^{G} A$. To prove this, we will construct inductively sequences of matrices $\{X_t\}_{t\ge 0}\subset GL(m, R)$,  $\{Y_t\}_{t\ge 0}\subset GL(n, R)$, and  a sequence of automorphisms $\{\varphi_t\}_{t\ge 1}\subset Aut(R)$ such that $\{X_t\cdot\varphi_t(A)\cdot Y_t\}_{t\ge 1}$  converges in the $\mathfrak{m}$-adic topology to $X\cdot\varphi(A)\cdot Y$ for some $X\in GL(m,R)$, $Y\in GL(n,R)$, and $\varphi\in Aut(R)$, and such that  
	$$B-X_t\cdot\varphi_t(A)\cdot Y_t\in \mathfrak{m}^{N+1+t}\cdot M_{m,n}$$ 
	holds for all $t\ge 1$. Then we obtain $B=X\cdot\varphi(A)\cdot Y.$
	
	For this, we first construct sequences of  matrices $ \{U_t\}_{t\ge 1}\subset GL(m,R)$, $ \{V_t\}_{t\ge 1}\subset GL(n,R)$, and  $\{A_t\}_{t\ge 0}\subset M_{m,n}$ with $A_0=A$ and a sequence of automorphisms $\{\phi_t\}_{t\ge 1}\subset Aut(R)$ such that for all $t\ge 1$, we have\\
	i) $A_t=U_t\cdot\phi_t(A_{t-1})\cdot V_t$\\
	ii) $B-A_t\in \mathfrak{m}^{N+t+1}\cdot M_{m,n}$.\\
	Setting $Q:=N-k\ge 1$ we have by assumption	
$$B-A\in \mathfrak{m}^{N+1}\cdot M_{m,n}=\mathfrak{m}^{Q-1}\cdot\mathfrak{m}^{k+2}\cdot M_{m,n}
	\subset \mathfrak{m}^Q\cdot\langle E_{m, pq}\cdot A\rangle+ \mathfrak{m}^Q\cdot\langle A\cdot E_{n, hl}\rangle +\mathfrak{m}^{Q+1}\cdot\left\langle \frac{\partial A}{\partial x_\nu}\right\rangle.$$
	Hence, there are $u_{pq}^{(1)}\in \mathfrak{m}^Q$, $v_{hl}^{(1)}\in \mathfrak{m}^Q$, and $d_{1,\nu}\in \mathfrak{m}^{Q+1}$ for all $p,q=1,\ldots,m$, $h,l=1,\ldots,n$, and $\nu=1,\ldots,s$ such that
	\begin{align*}
	B-A&=
	\sum\limits_{p,q = 1}^m {u_{pq}^{(1)}\cdot E_{m, pq} }\cdot A+\sum\limits_{h,l = 1}^n {v_{hl}^{(1)} \cdot A\cdot E_{n, hl} }+
	\sum\limits_{\nu = 1}^s {d_{1,\nu}\cdot \frac{{\partial A}}{{\partial x_\nu }}}\\
	&= U^{(1)}\cdot A+A\cdot V^{(1)}+\sum\limits_{\nu= 1}^s {d_{1,\nu}\cdot \frac{{\partial A}}{{\partial x_\nu}}},
	\end{align*}
	\noindent where $U^{(1)}:=[u_{pq}^{(1)}]\in \mathfrak{m}^Q\cdot M_m$ and $V^{(1)}:=[v_{hl}^{(1)}]\in \mathfrak{m}^Q\cdot M_n.$
	
	Define $U_1:=I_m+U^{(1)}\in GL(m,R)$, $V_1:=I_n+V^{(1)}\in GL(n,R)$,  
	$$\phi_1: R \to R, \hskip 9pt x_\nu\mapsto x_\nu+d_{1,\nu}\hskip3pt, \hskip 4pt  \nu=1,\ldots,s,$$
	and $A_1:=U_1\cdot\phi_1(A)\cdot V_1$.
	We now show that 
	$$B-A_1\in\mathfrak{m}^{N+2}\cdot M_{m,n}.$$
	For all $i=1,\ldots, m$ and $j=1,\ldots,n$, applying Lemma \ref{expansion} to $f=a_{ij}=\sum\limits_{\left| \alpha  \right| \ge o_{ij}} {c_\alpha^{(ij)}  {\bf x}^\alpha  }$ and $(z_1,\ldots,z_s)=(d_{1,1},\ldots, d_{1,s})$ we have 
	$$\phi_1(a_{ij}({\bf x}))=a_{ij}(x_1+d_{1,1},\ldots, x_s+d_{1,s})=a_{ij}({\bf x})+
	\sum\limits_{\nu = 1}^s {\frac{{\partial a_{ij}({\bf x})}}{{\partial x_\nu}}\cdot d_{1,\nu} } +h_{ij},$$
	where
	$$h_{ij}=\sum\limits_{\left| \alpha  \right| \ge o_{ij}} {{c_\alpha^{(ij)}\cdot  \left( {\sum\limits_{ | \gamma| \ge 2  \atop 
					\gamma  \le \alpha } {\binom
					{\alpha _1 } { \gamma _1} \cdot\ldots\cdot\binom
					{\alpha _s } { \gamma _s}{\bf x}^{\alpha  - \gamma }\cdot d_{1,1}^{\gamma_ 1}\cdot\ldots \cdot d_{1,s}^{\gamma _s} } }\right)}}.
	$$ 
	Then for all $i$ and $j$, for $\alpha\in\N^s$, $|\alpha|\ge o_{ij}$, and for $\gamma\in \N^s$, $\gamma\le\alpha$, $|\gamma|\ge 2$, we have 
	$$|\alpha|-|\gamma|+(Q+1)|\gamma|=|\gamma|Q+|\alpha|\ge 2Q+o_{ij}\ge 2Q+o=N+2.$$
	Thus, for all $i=1,\ldots,m$ and $j=1,\ldots,n$, we have $h_{ij}\in\mathfrak{m}^ {N+2}.$
	We obtain 
	$$\phi_1(A)=A+\sum\limits_{\nu=1}^sd_{1,\nu}\cdot\frac{\partial A}{\partial x_\nu}+H,$$
	where $H=[h_{ij}]$ and $H\in\mathfrak{m}^{N+2}\cdot M_{m,n}$. This implies that
	\begin{align*}
	B-A_1=B&-\left(I_m+U^{(1)}\right)\cdot\phi_1(A)\cdot\left(I_n+V^{(1)}\right)= -U^{(1)}\cdot A\cdot V^{(1)} - U_1\cdot H\cdot V_1 -\\
	&-{\sum\limits_{\nu  = 1}^s {d_{1,\nu }\cdot U^{(1)}\cdot\frac{{\partial A}}{{\partial x_\nu  }}} } - {\sum\limits_{\nu  = 1}^s {d_{1,\nu }\cdot \frac{{\partial A}}{{\partial x_\nu  }}} }\cdot V^{(1)}-  {\sum\limits_{\nu  = 1}^s {d_{1,\nu }\cdot U^{(1)}\cdot\frac{{\partial A}}{{\partial x_\nu  }}} }\cdot V^{(1)}\\
	&\in \mathfrak{m}^{N+2}\cdot M_{m,n}.
	\end{align*}

	Now that by assumption and Lemma {\ref{lemma}} 
	\begin{align*}
	\mathfrak{m}^{k+2}\cdot M_{m,n} \subset \mathfrak{m}\cdot\langle E_{m, pq}\cdot A_1\rangle+\mathfrak{m}\cdot\langle A_1\cdot E_{n, hl}\rangle+\mathfrak{m}^2\cdot\left\langle \frac{\partial (A_1)}{\partial x_\nu}\right\rangle.
	\end{align*}

	Furthermore, since $\order(A_1)=\order(A)=o$, we can proceed inductively to construct the sequences $\{A_t\}_{t\ge 0}$, $\{U_t\}_{t\ge 1}$, $\{V_t\}_{t\ge 1}$, and $\{\phi_t\}_{t\ge 1}$ as desired.
	
	%%%%%%%%%%%%%%
	Now, for $t\ge 1$, we define 
	\begin{align*}
	&\varphi_t:=\phi_t\circ\ldots\circ\phi_1,\\
	&X_t=U_t\cdot\phi_t(X_{t-1}),\hskip 5pt X_0=I_m,  {\text {and}}\\
	&Y_t=\phi_t(Y_{t-1})\cdot V_t, \hskip 14pt  Y_0=I_n.
	\end{align*}
	Then by induction we obtain $A_t=X_{t}\cdot\varphi_t(A)\cdot Y_t$ and $B-X_t\cdot\varphi_t(A)\cdot Y_t\in \mathfrak{m}^{N+t+1}\cdot M_{m,n}.$

	It remains to prove that $\big\{X_t\cdot\varphi_t(A)\cdot Y_t\big\}_{t\ge 1}$ converges to $X\cdot\varphi(A)\cdot Y$ in the $\mathfrak{m}$-adic topology for some $\varphi\in Aut(R)$, $X\in GL(m,R)$, and $Y\in GL(n,R)$. For that, we show that the sequences $\{X_t\}_{t\ge 0}$, $\{Y_t\}_{t\ge 0}$, and $\{\varphi_t(x_\nu)\}_{t\ge 1}$ for all $\nu\in\{1,\ldots, s\}$ are Cauchy sequences and define $X$, $Y$, and $\varphi$ as their limits, respectively. 
	We have for all $t\ge 1$,
	$$X_t-X_{t-1}=\left(I_m+U^{(t)}\right)\cdot\phi_t(X_{t-1})-X_{t-1}=\phi_t(X_{t-1})-X_{t-1}+U^{(t)}\cdot\phi_t(X_{t-1})\in \mathfrak{m}^{Q+t-1}\cdot M_{m}$$
	since $\phi_t(X_{t-1})-X_{t-1}\in \mathfrak{m}^{Q+t}\cdot M_{m}$ and $U^{(t)}\in \mathfrak{m}^{Q+t-1}\cdot M_{m}$. \\

Hence, given $P\ge 1$, for $t>r\ge r_0$ with $r_0=\max\{P-Q, 1\}$ we have
	$$X_t-X_r=\left(X_t-X_{t-1}\right) + \ldots +\left(X_{r+1}-X_{r}\right)\in \mathfrak{m}^{Q+r}\cdot M_{m}\subset\mathfrak{m}^{P}\cdot M_{m}.$$
	This shows that 
	$\{X_t\}_{t\ge 0}$ is a Cauchy sequence in $M_{m}$, and thus it converges to a matrix $X\in M_m$. Moreover, it follows by induction that
	$$X_t-I_m\in \mathfrak{m}^Q\cdot M_{m}$$
	so that $X=I_m+X_0$ for some $X_0\in\mathfrak{m}^Q\cdot M_{m}.$ By the same argument, $\{Y_t\}_{t\ge 0}$ converges to a $Y\in GL(n,R)$. Now fix $\nu\in \{1,\ldots,s\}$ and  express $\phi_t(\varphi_{t-1}(x_\nu))$ in term of $\varphi_{t-1}(x_\nu)$ as above, we have 
	$$\varphi_t(x_\nu)-\varphi_{t-1}(x_\nu)
	=\phi_t(\varphi_{t-1}(x_\nu))-\varphi_{t-1}(x_\nu)\in\mathfrak{m}^{Q+t}$$
	since $d_{t,\nu}\in \mathfrak{m}^{Q+t}$ for all $\nu=1,\ldots,s$. By a similar argument as above, 
	$\{\varphi_t(x_\nu)\}_{t\ge 1}$ is a Cauchy sequence and hence converges in $R$. In addition, by induction we have 
	$$\varphi_t(x_\nu)- x_\nu=\varphi_{t}(x_\nu)-\varphi_{t-1}
	(x_\nu)+\varphi_{t-1} (x_\nu)-x_\nu\in \mathfrak{m}^{Q+1}.$$
	This implies that $\{\varphi_t(x_\nu)\}_{t\ge 1}$ converges to a power series $x_\nu+d_\nu$ for some $d_\nu\in \mathfrak{m}^{Q+1}$. Define the automorphism $\varphi$ by
	$$\varphi: R \to R, \hskip 6pt x_\nu \mapsto x_\nu+d_\nu, \hskip 6pt \nu=1,\ldots, s.$$
	 
	Then for all $i=1,\ldots, m$ and $j=1,\ldots,n $,
	$$\varphi_t(a_{ij})-\varphi(a_{ij})\in \mathfrak{m}^P,$$
	which implies $\varphi_t(A)-\varphi(A)\in \mathfrak{m}^P\cdot M_{m,n}$. On the other hand, since $\{X_t\}_{t\ge 0}$ converges to $X$ and $\{Y_t\}_{t\ge 0}$ converges to $Y$, there are  $t_1$ and $t_2\in \N$ such that $X_t-X\in \mathfrak{m}^P\cdot M_{m}$ for all $t\ge t_1$ and $Y_t-Y\in \mathfrak{m}^P\cdot M_{n}$ for all $t\ge t_2$. Hence, for $t\ge t_3$ with $t_3:=\max\{t_0, t_1, t_2\}\ge1$, we have
	\begin{align*}
	X_t\cdot\varphi_t(A)\cdot Y_t-X\cdot\varphi(A)\cdot Y&=X_t\cdot\varphi_t(A)\cdot(Y_t-Y)+X_t\cdot(\varphi_t(A)-\varphi(A))\cdot Y\\
	&\hskip 15pt +(X_t-X)\cdot\varphi(A)\cdot Y\in\mathfrak{m}^P\cdot M_{m,n}.
	\end{align*}
By uniqueness of the limit, we get $B=X\cdot\varphi(A)\cdot Y.$
\end{proof}

\begin{Remark}\rm
	\begin{enumerate}
		\item For $\characteristic(K)=0$, condition \eqref{main} is also necessary for finite determinacy, see Proposition \ref{separable criterion}. We do not know whether this is true in positive characteristic for arbitrary $m$, $n$. For $n=1$ and $G=\mathcal{G}_{lr}$ this is true as shown in \cite{GP16} (here $\mathcal{G}_{lr}$-equivalence coincides with $\mathcal{G}_{l}$-equivalence).
		\item The determinacy bound given in Theorem \ref{finite determinacy} is new also in characteristic zero, but we expect that it can be improved by using integration of vector fields.
	\end{enumerate}
\end{Remark}

% % % % % % % % % % % % % % % % % % % % % % % % % % % % % % % % % % % %
\section{Finite determinacy for separable orbit maps}\label{sufficient condition 2}

We provide a criterion saying that for a matrix $A\in \mathfrak{m}\cdot M_{m,n}$, if the orbit map $G^{(k)}\to G^{(k)}jet_k(A)$ is separable for all $k\ge k_0$, some $k_0\in \N$, then finite $G$-determinacy and finite codimension of the tangent image $\widetilde T_A(GA)$ are equivalent.

We derive first equivalent conditions to finite codimension of $\widetilde T_A(GA)$ for which we need the following result from commutative algebra. 

\begin{Lemma}{\label{General lemma}}
	Let $(R,\mathfrak{m})$ be a local Noetherian $K$-algebra and $L$ a finitely generated $R$-module. Then the following are equivalent:
	\begin{enumerate}
		\item  $\dim_KL<\infty$.
		\item $\mathfrak{m}^k\cdot L=0$ for some $k$.
		\item Let $R^t\xrightarrow{\Theta} R^l  \to L \to 0$ be a presentation of $L$, then there is some $h$ such that $\mathfrak{m}^h \subset I_l(\Theta)$, where $I_l(\Theta)$ is the 0-th Fitting ideal of $L$, i.e.  the ideal generated by all $l\times l$ minors of the presentation matrix $\Theta$.
	\end{enumerate}
	Moreover, with the same notation as in 3., $L=0$ if and only if $I_l(\Theta)=R$.
	
	{\bf Supplement:} If $\mathfrak{m}^h \subset I_l(\Theta)$ for some positive integer $h$ then $\mathfrak{m}^h\cdot L=0$. 
\end{Lemma}

\begin{proof}
 Follows easily from \cite[Proposition 20.7]{Eis95}.
\end{proof}

\begin{Proposition}\label{sufficient condition for determinacy}
	Let $A\in \mathfrak{m}\cdot M_{m,n}$. Then $A$ is finitely $G$-determined if one of the following  equivalent statements holds:
	\begin{enumerate}
		\item\label{sufficient condition for determinacy 1}  $\mathfrak{m}^{l+1}\cdot M_{m,n}\subset \widetilde T_A(GA)$ for some positive integer $l$.
		\item\label{sufficient condition for determinacy 2}  $\dim_K\left(\mathfrak{m}\cdot M_{m,n}/\widetilde T_A(GA)\right) =:d<\infty$.
		\item\label{sufficient condition for determinacy 3} $\dim_K\left(M_{m,n} \big{/}\widetilde T_A^e(GA)\right) =:d_e<\infty.$
		\item\label{sufficient condition for determinacy 4} $\mathfrak{m}^h\cdot M_{m,n}\subset \widetilde T_A^e(GA)$ for some positive integer $h$.
		\item\label{sufficient condition for determinacy 5}  $\mathfrak{m}^k \subset I_{mn}\left(\Theta_{(G,A)}\right)$ for some positive integer $k$, where $$R^t\xrightarrow{\Theta_{(G,A)}} M_{m,n}  \to M_{m,n}/\widetilde T_A^e(GA)\to 0$$
		is a presentation of $M_{m,n}/\widetilde T_A^e(GA)$, i.e.  $\variety\left(I_{mn}\left(\Theta_{(G,A)}\right)\right)= \{\mathfrak{m}\}$.
		\item $\support\left(M_{m,n}/\widetilde T_A^e(GA)\right)=\{\mathfrak{m}\}$, i.e. $\left(M_{m,n}\right)_P=\left(\widetilde T_A^e(GA)\right)_P$ for all $P\in \spectrum(R)\smallsetminus \{\mathfrak{m}\}$.
	\end{enumerate}
	Furthermore, if the condition 1. (resp. 2., 3., 4., and 5.)  above holds  then $A$ is $G$ ($2c-\order(A)+2$)-determined, where $c= l$ (resp. d, $d_e$, h, and k).
\end{Proposition}

\begin{proof} By Theorem \ref{finite determinacy} condition 1. implies finite determinacy. We prove that the six conditions are equivalent:
	
	(1. $\Rightarrow$ 2.) This is clear since $\dim_K\left(\mathfrak{m}\cdot M_{m,n}/\mathfrak{m}^{l+1}\cdot M_{m,n}\right)<\infty$.
	
	(2. $\Rightarrow$ 3.) Since $\widetilde T_A(GA)\subset \widetilde T_A^e(GA),$ we have $$\dim_K\left(M_{m,n} \big{/}\widetilde T_A^e(GA)\right)\le \dim_K\left(M_{m,n} \big{/}\widetilde T_A(GA)\right)<\infty.$$
	Now use that the difference between $\dim_K\left(M_{m,n} \big{/}\widetilde T_A(GA)\right)$ and $\dim_K\left(\mathfrak{m}\cdot M_{m,n}/\widetilde T_A(GA)\right)$ is finite.
	
	(3. $\Rightarrow$ 4.) We have a chain of $K$-vector subspaces:
	$$M_{m,n}\big{/}\widetilde T^e_A(GA)\supset \left(\mathfrak{m}\cdot M_{m,n}+\widetilde T^e_A(GA)\right)\big{/}\widetilde T^e_A(GA)\supset \ldots $$
	If $\dim_K\left(M_{m,n}/\widetilde T^e_A(GA)\right)=d_e$ then $\left(\mathfrak{m}^{d_e}\cdot M_{m,n}+\widetilde T^e_A(GA)\right)\big{/}\widetilde T^e_A(GA)=0$ so that $\mathfrak{m}^{d_e}\cdot M_{m,n}\subset \widetilde T^e_A(GA)$.
	
	(4. $\Rightarrow$ 1.) This is obvious.
	
	(4. $\Leftrightarrow$ 5.) Apply Lemma \ref{General lemma} (2. $\Leftrightarrow$ 3.) to $L=M_{m,n}/\widetilde T^e_A(GA)$.
	
	$(5.\Leftrightarrow 6.)$. Let $R^t\xrightarrow{\Theta_{(G,A)}} M_{m,n}  \to M_{m,n}/\widetilde T_A^e(GA) \to 0$ be a presentation of $M_{m,n}/\widetilde T_A^e(GA)$. By \cite[Proposition 20.7]{Eis95}, $$\sqrt {I_{mn}\left(\Theta_{(G,A)}\right)} = \sqrt{\annihilator_R\left(M_{m,n}/\widetilde T_A^e(GA)\right)}.$$
	This implies the zero sets of the two ideals $I_{mn}\left(\Theta_{(G,A)}\right)$ and $\annihilator_R\left(M_{m,n}/\widetilde T_A^e(GA)\right)$ define the same varieties in $\spectrum(R)$. That means 
	$$\support\left(M_{m,n}/\widetilde T_A^e(GA)\right)=\variety\left(I_{mn}\left(\Theta_{(G,A)}\right)\right)$$
	and thus the two assertions are equivalent. 
	
	To derive the determinacy bound, we apply Theorem \ref{finite determinacy}. By multiplying the inclusion in 1. (resp. 4) with $\mathfrak{m}$ (resp. $\mathfrak{m}^2$) and  applying Theorem \ref{finite determinacy}, we get the bound. Assume that the condition in 3. (resp. 2.) holds. Using the argument as in the proof of  $(3.\Leftrightarrow 4.)$ we obtain that $\mathfrak{m}^{d_e}\cdot M_{m,n}\subset \widetilde T^e_A(GA)$ (resp. $\mathfrak{m}^{d+1}\cdot M_{m,n}\subset \widetilde T_A(GA)$). Applying the determinacy bound for 4. (resp. 1.), we get the claim. Now assume that condition 5. holds. Apply the supplement of Lemma {\ref {General lemma}} to $L= M_{m,n}/\widetilde T_A^e(GA)$, we have $\mathfrak{m}^k\cdot L=0$. This implies $\mathfrak{m}^k\cdot M_{m,n}\subset \widetilde T_A^e(GA)$,  and thus  $A$ is $G$ $(2k-\order(A)+2)$-determined. 
\end{proof}

\begin{Theorem}\label{separable criterion}
	Let $A\in \mathfrak{m}\cdot M_{m,n}$. 
	\begin{enumerate}
	\item Assume that $A$ is finitely $G$-determined and the orbit map $G^{(l)}\to G^{(l)}jet_l(A)$ is separable for some $l> k$, where $k$ is a $G$-determinacy bound of $A$. Then  the equivalent conditions in Proposition \ref{sufficient condition for determinacy} hold. 	
	\item Assume that there is some $k\in \N$ such that the orbit map  $G^{(l)}\to G^{(l)}jet_l(A)$ is separable for all $l\ge k$ (e.g. if $\characteristic(K)=0$). Then $A$ is finitely $G$-determined if and only if one of the equivalent conditions in Proposition \ref{sufficient condition for determinacy} holds.
	\end{enumerate}
\end{Theorem}

The criterion of the following corollary was proved in \cite{Mat68} for complex or real analytic and $\C^\infty$ map germs, i.e. $M_{m,1}$. It follows for arbitrary matrices from Theorem \ref {separable criterion}, Theorem \ref{finite determinacy} and Lemma \ref {tangent space char 0}  (and was stated in \cite[1.4]{BK16} without proof):

\begin{Corollary}\label{Mather's criterion}
	If $\characteristic(K)=0$ then the matrix $A$ is finitely $G$-determined if and only if the embedding  $T_A(GA)\hookrightarrow M_{m,n}$ is of finite codimension, i.e.
	$$\dim_{K}\left( M_{m,n}/T_A(GA)\right)<\infty.$$

\end{Corollary}

\begin{proof} (of the theorem)
    1. We prove that 
    $$\mathfrak{m}^{l}\cdot M_{m,n}\subset \widetilde T_A(GA).$$
	Indeed, let $B\in \mathfrak{m}^{l}\cdot M_{m,n}$.  Since $A$ is $G$ $(l-1)$-determined,
	$A+t\cdot B\in GA$ for all $t\in K$
	so that
	$$jet_{l}(A)+t \cdot jet_{l}(B)\in G^{(l)}jet_{l}(A)$$
	for all $t\in K$. This implies 
	$$jet_{l}(B)\in T_{jet_{l}(A)}\left(G^{(l)}jet_{l}(A)\right)=\widetilde T_{jet_{l}(A)}\left(G^{(l)}jet_{l}(A)\right),$$
	where the equality follows by the separability assumption using Corollary \ref {Corollary2.3}. Let $p_{l}: M_{m,n}\to  M_{m,n}^{(l)}$ be the projection. Then by taking the preimages, we obtain
	$$B\in \widetilde T_A(GA)+\mathfrak{m}^{l+1}\cdot M_{m,n}.$$
	This shows that 
	$$\mathfrak{m}^{l}\cdot M_{m,n}\subset \widetilde T_A(GA)+\mathfrak{m}^{l+1}\cdot M_{m,n},$$
	and implies that
	$$\mathfrak{m}^{l}\cdot M_{m,n}\subset  \widetilde T_A(GA)\cap\mathfrak{m}^{l}\cdot M_{m,n}+
	\mathfrak{m}^{l+1}\cdot M_{m,n}.$$
	Applying Nakayama's lemma to the $R$-submodule $ \widetilde T_A(GA)\cap\mathfrak{m}^{l}\cdot M_{m,n}$ 
	of the $R$-module $\mathfrak{m}^{l}\cdot M_{m,n}$ we obtain
	$$\mathfrak{m}^{l}\cdot M_{m,n}\subset \widetilde T_A(GA).$$
	
	2. Follows from 1. and Proposition 	\ref{sufficient condition for determinacy}.
\end{proof}

	We finish with a result for $G=\mathcal{R} = Aut(K[[x_1,..., x_s]])$, i.e. about right equivalence, by showing that finite determinacy is rather rare for $m>1$ under the separability condition. Let $A\in \mathfrak{m}\cdot M_{m,n}$. Since the action of $G$ does not use the matrix structure, we may assume that $A=[a_1\hskip 5pt a_2\ldots a_m]^T\in M_{m,1}$ is the matrix of one column, i.e. without loss of generality we may assume that $n=1$, just to make the notation shorter. Let $Jac(A):=\left[ \frac{\partial a_i}{\partial x_j}\right]\in Mat(m,s,R)$ be the Jacobian matrix of the vector $(a_1,\ldots, a_m)\in R^m$. Then $Jac(A)$ is a presentation matrix of $M_{m,1}/\widetilde T^e_A(GA) = R^m/\langle \frac{\partial A}{\partial x_j}\rangle $ and we have:

\begin{Proposition} \rm Let $G=\mathcal{R}$ and $A \in \mathfrak{m}\cdot M_{m,1}$. 
	\begin{enumerate}
	\item If $m>s$ then 
	$$\dim_K\left(M_{m,1}/\widetilde T_A(GA)\right)=\infty.$$ 
	If in addition there is some $k\in \N$ such that the orbit map  $G^{(l)}\to G^{(l)}jet_l(A)$ is separable for all $l\ge k$, then $A$ is not finitely right determined.  	
	\item If $m\le s$ and $\{\overline{a_i}\}$  is linearly independent in $\mathfrak{m}/\mathfrak{m}^2$, then $A$ is finitely right determined. If $m>1$ and there is some $k\in \N$ such that the orbit map  $G^{(l)}\to G^{(l)}jet_l(A)$ is separable for all $l\ge k$, then the converse also holds. 

\end{enumerate}
\end{Proposition}

\begin{proof}
	1. Since $I_m(Jac(A))=\{0\}$, by Lemma \ref{General lemma} $\widetilde T_A(GA)$ has infinite codimension. If the assumption on separability is satisfied then by Theorem \ref{separable criterion}, $A$ is not finitely right determined.
	
	2. 	Indeed, we prove that $A$ is right 1-determined in the first assertion. Let $B=[b_1\hskip 5pt b_2\ldots b_m]^T\in M_{m,1}$ be such that $c_i:=b_i-a_i\in \mathfrak{m}^2$ for all $i=1,\ldots, m$. Let $a_{m+1},\ldots, a_s\in \mathfrak{m}$ be such that $\{\overline{a_i}\}$, $i=1,\ldots, s$ is a basis of $\mathfrak{m}/\mathfrak{m}^2$. Let
	$\phi_1, \phi_2: R\to R$ be automorphisms defined by $\phi_1(x_\nu)=a_\nu$ for $\nu=1,\ldots,s$,
	and $\phi_2(x_\nu)= a_\nu   + c_\nu$ for $\nu  = 1,\ldots,m$ and $\phi_2(x_\nu)=a_\nu$  for $\nu  = m + 1,\ldots,s$. 
	Then  $
	\phi : = \phi _2  \circ \phi _1 ^{ - 1}\in \mathcal{R}$ and $\phi(A)=B$. For the second statement, assume that $A$ is finitely right determined. By the assumption on separability of the orbit maps and the statement 1. we have $m\le s$. Moreover, assume by contradiction that $\{\overline{a_{i}}\}$ is linearly dependent in $\mathfrak{m}/\mathfrak{m}^2$. Then 
	$\rank
	\left[ {\frac{{\partial a_{i} }}{{\partial x_j }}({\bf 0})}\right]<m,
	$
	where ${\frac{{\partial a_{i} }}{{\partial x_j }}({\bf 0})}$ is the constant term of ${\frac{{\partial a_{i} }}{{\partial x_j}}}$. This implies that $I_m(Jac(A))$ is a proper ideal of $R$, and by \cite[Theorem 13.10]{Mat80}, this ideal has the height at most $s-m+1$. This yields
	$$\dim \left(R/I_m(Jac(A))\right)\ge m-1>0$$
	so that $I_m(Jac(A))$ can not contain any power of $\mathfrak{m}$. By Proposition \ref{separable criterion}, $A$ is not finitely right determined, a contradiction.
\end {proof}

\begin{Corollary} \rm
    If $\characteristic(K)=0$ we have
    \begin{enumerate}
    	\item[(i)] If $m>s$ then $A$ is not finitely right determined.
    	\item[(ii)] Let $m\le s$. If $m>1$ then $A$ is finitely right determined if and only if $\{\overline{a_i}\}$ is linearly independent in $\mathfrak{m}/\mathfrak{m}^2$. 
    \end{enumerate}
\end{Corollary}

\begin {Remark} \rm
 If $m=1$ and the characteristic of $K$ is arbitrary, then $A=[a]$ is finitely right determined if and only if $M_{m,n}/\widetilde T^e_A(GA)=R\big{/}\left\langle \frac{\partial a}{\partial x_1},\ldots,\frac{\partial a}{\partial x_s} \right\rangle$ is finite dimensional over $K$.
This is well known if $\characteristic(K)=0$. The case of positive characteristic was stated in \cite {BGM12} but the proof contains a gap, which is closed in \cite {GP16}.
\end{Remark}

\vskip 7pt 
{\bf Acknowledgement:} We thank Gerhard Pfister for useful discussions. The second author would like to thank the DAAD (Germany) for support.

%%%%%%%%%%%%%%%%%%%%%%%%%%%%%%%%%%%%%%%%%%%%%%
%\nocite{*}
%\bibliographystyle{amsalpha}
%\bibliography{Phamthuyhuong1}

\providecommand{\bysame}{\leavevmode\hbox to3em{\hrulefill}\thinspace}
\providecommand{\MR}{\relax\ifhmode\unskip\space\fi MR }
\providecommand{\MRhref}[2]{%
	\href{http://www.ams.org/mathscinet-getitem?mr=#1}{#2}
}
\providecommand{\href}[2]{#2}

Fachbereich Mathematik, Universit\"at Kaiserslautern, Erwin-Schr\"odinger Str.,
67663 Kaiserslautern, Germany

E-mail address: greuel@mathematik.uni-kl.de

\vskip 7pt 

Department of Mathematics, Quy Nhon University, 170 An Duong Vuong
Street, Quy Nhon City, Vietnam

Email address: phamthuyhuong@qnu.edu.vn
\end{document}